\documentclass[12pt]{amsart}
\usepackage{latexsym,fancyhdr,amssymb,color,amsmath,amsthm,graphicx,listings,comment}
\usepackage[section]{placeins}  \newtheorem{thm}{Theorem} 
 \newtheorem{propo}{Proposition}
\newtheorem{coro}{Corollary}  \let\paragraph\subsection
\def\Binomial#1#2{{#1\choose #2}}

\title{Gauss-Bonnet for Form Curvatures}
\author{Oliver Knill}
\date{8/28/2024}
\address{Department of Mathematics \\ Harvard University \\ Cambridge, MA, 02138 }
\subjclass{}
\keywords{Gauss Bonnet}
\begin{document}
\maketitle

\begin{abstract}
We look at curvatures that are supported on k-dimensional
parts of a simplicial complex $G$. These curvature all satisfy the 
Gauss-Bonnet theorem, provided that the k-dimensional simplices cover $G$.
Each of these curvatures can be written as an expectation of Poincare-Hopf indices. 
Linear or non-linear wave dynamics with discrete or continuous time 
allow to deform these curvatures while keeping the Gauss-Bonnet property. 
\end{abstract}

\section{Introduction}

\paragraph{}
We start our story with a familiar special case
of a convex polyhedron $G$. If $G$ has $|V|$ vertices, $|E|$ edges and $|F|$ triangular faces,
the {\bf Euler polyhedron formula} $\chi(G)=|V|-|E|+|F|=2$ holds. Discovered empirically first
by Descartes \cite{Aczel}, then proven by Euler in more generality what we call today 
{\bf planar graphs} \cite{Euler1758}. The polyhedral case needed time to stabilize.
\cite{lakatos} made it a case study in the epitemology of sciences. 
The book of \cite{Richeson} illustrates the drama. 
Various combinatorial notions of curvature have been suggested. First maybe
by Victor Eberhard (shown in Figure~\ref{eberhard}) \cite{Eberhard1891}. The curvature
$K(v)=1-d(v)/6$ with vertex degree $d(v)$ was used already 
in graph coloring contexts \cite{Heesch}. If all faces are triangular and no boundary
is present, then the {\bf Gauss-Bonnet formula} $\chi(G) = \sum_{v \in V} K(v)$ hold.
To prove Gauss-Bonnet think about $\omega(x) = (-1)^{{\rm dim}(x)}$ at the "energy"
\footnote{\cite{KnillEnergy2020} shows in full generality that $\chi(G)$ is always the 
sum of potential energies $g(x,y)$, where $g$ is the inverse of the connection Laplacian,
which is always unimodular. In physics, Green functions give potential values.}
of $x \in V \cup E \cup F$. Euler characteristic $\chi(G)=\sum_{x \in G} \omega(x)$
is the total energy. Move the energy from every $x$ equally to all vertices in $x$:
every edges gives $-1/2$ to its vertices, every triangle gives $1/3$ to its vertices. 
$1$ ends up at the vertex. Then $K(v)=1-d(v)/2+d(v)/6=1-d(v)/6$. 
More references are \cite{Gromov87,Presnov1991,Higuchi,NarayanSaniee,princetonguide}.

\paragraph{}
Combinatorial curvatures exist for any simplicial complex and especially all
complexes that come from graphs, where the k-simplices are complete sub-graphs $K_{k+1}$.
The {\bf Levitt curvature} \cite{Levitt1992} on vertices $v \in V=\{ x \in G, |x|=1\}$ is
$$ K(v) = \sum_{x \in G, v \in x} \frac{\omega(x)}{|x|} $$
with $\omega(x) = (-1)^{{\rm dim}(x)}$. It is a vertex curvature 
which works on a general {\bf finite abstract simplicial complex} with Euler characteristic
$\chi(G)=\sum_{x \in G} \omega(x)$. Gauss Bonnet means $\chi(G)=\sum_{v \in V} K(v)$. 
\footnote{Rediscovered in \cite{cherngaussbonnet} 
while investigating curvatures for higher dimensional manifolds.
We found the reference \cite{Levitt1992} in 2015, while working on 
multi-linear valuations \cite{valuation}.}
As already noted in \cite{cherngaussbonnet}, the proof is a simple 
conservation principle in that the energy $\omega(x)$ on $x$ 
is distributed to the $|x|=k+1$ vertices of a $k$-dimensional simplex $x$. 
We generalize this slightly here and sow that 
the energies $\omega(x)$ can also be distributed to $k$-dimensional parts 
$G_k = \{ x \in G, |x|=k+1\}$ of $G$, producing curvatures $K_k$. 
In \cite{cherngaussbonnet} experiments showed $K_0(v) = 0$ for
odd-dimensional manifolds. It was proven using index expectation \cite{indexexpectation}.
Curvatures can be seen as index expectation of {\bf Poincar\'e-Hopf
indices} $i_f$, where $f$ are functions on facets. If $i_f=-i_{-f}$, the sum is zero if
the probability measure invariant under the involution $f \to -f$.

\begin{figure}[!htpb]
\scalebox{0.07}{\includegraphics{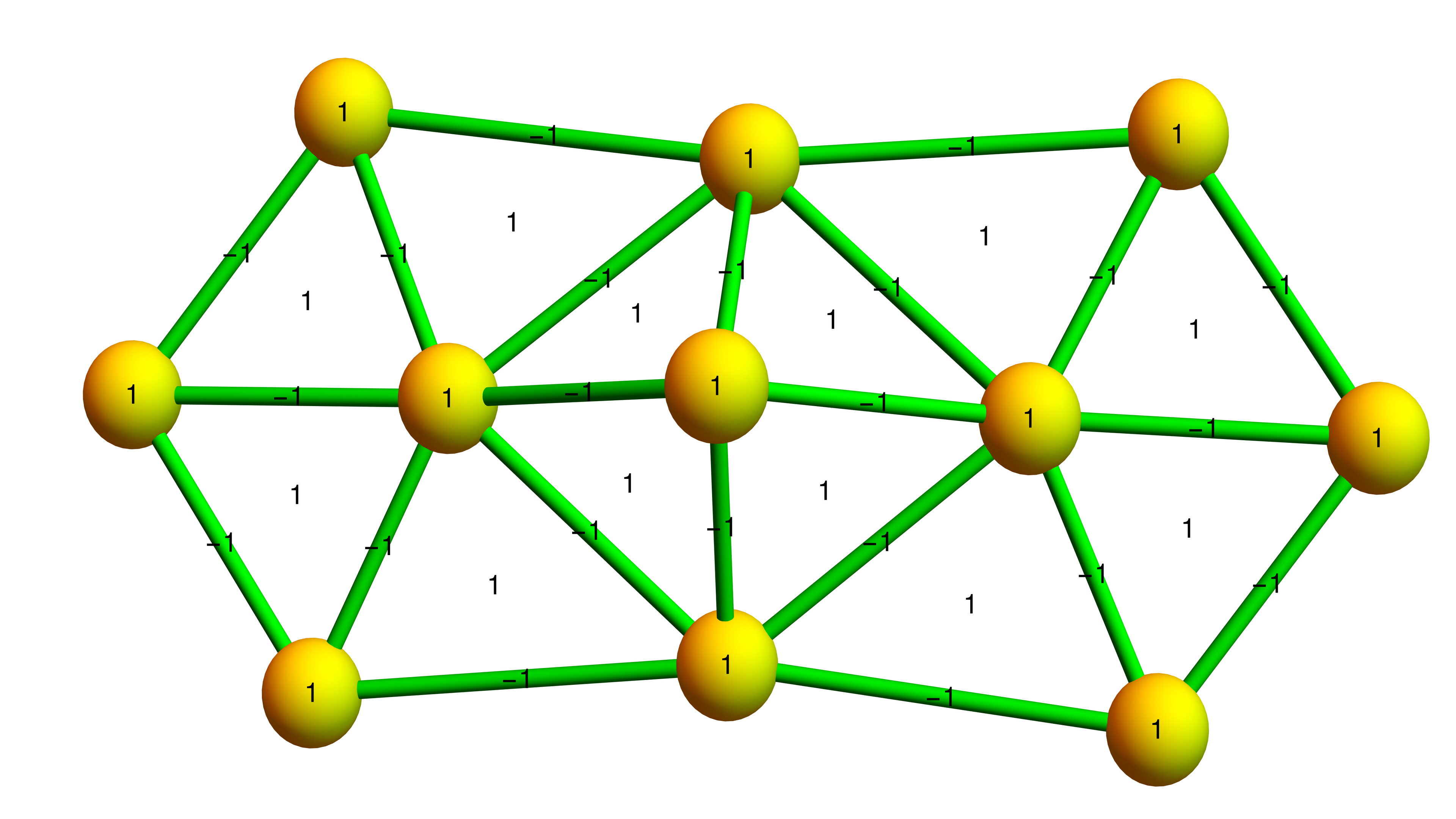}}
\scalebox{0.07}{\includegraphics{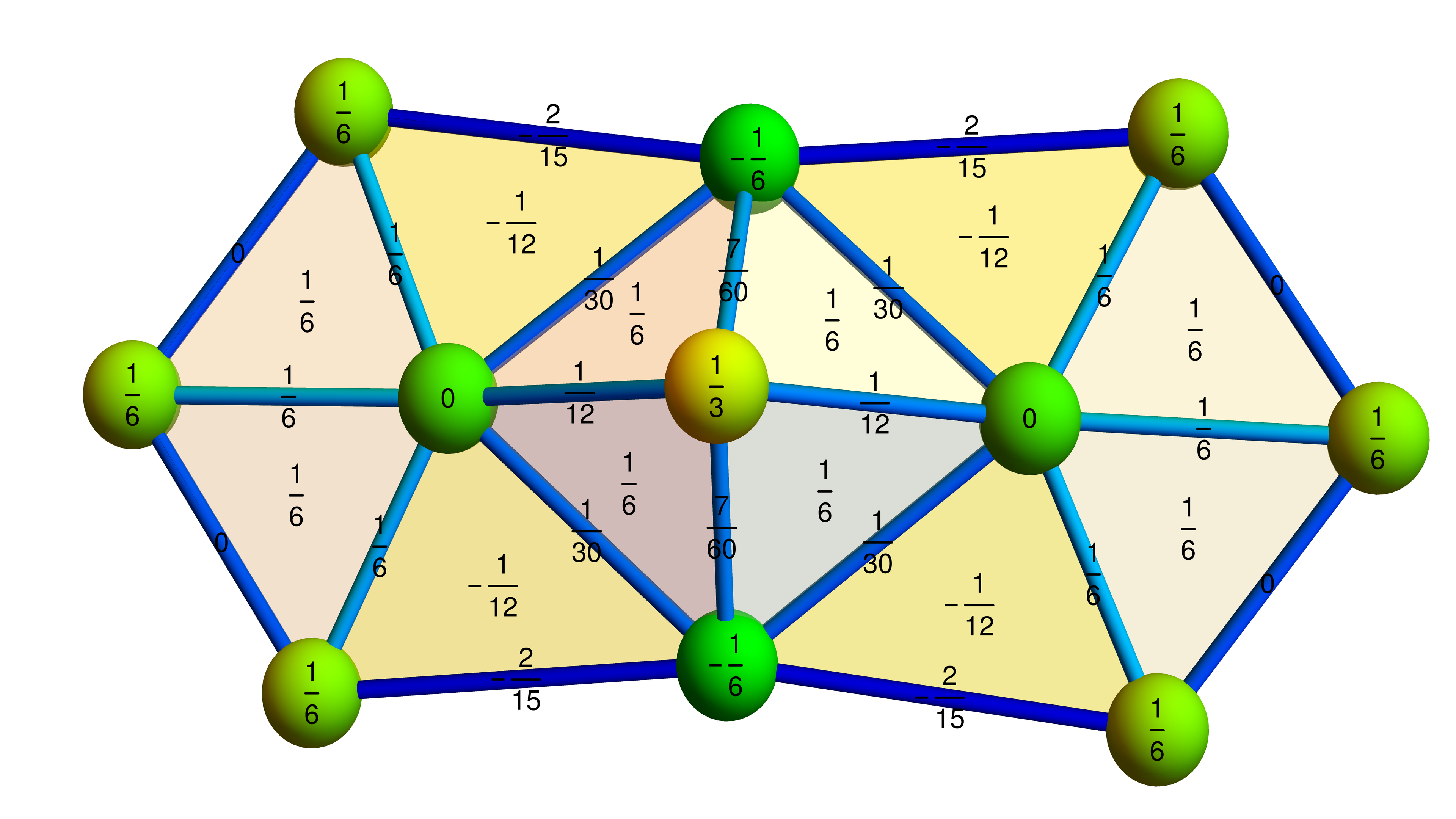}}
\label{kite}
\caption{
We see a two-dimensional region with $11$ vertices $V$, $22$ edges $E$ and $12$ 
faces $F$, where $V-E+F=1$. The original energies $1$ on $V \cup F$ and $-1$ on $E$ can be 
distributed to vertices, edges and faces to get curvatures
$K(v)$, $K(e)$ and $K(f)$ each satisfying Gauss-Bonnet
$\chi(G)=\sum_{v \in V} K(v) = \sum_{e \in E} K(e) = \sum_{f \in F} K(f)$. 
}
\end{figure}

\paragraph{}
If the energies $\omega(x)$ of a 2 dimensional surface without boundary 
are moved to edges $e=(a,b)$, we end up with 
$K(e) = 1/d(a) + 1/d(b) -1 + 2/3 = 1/d(a) + 1/d(b)-1/3$. If the energies 
are moved to the triangles $f=(a,b,c)$, we end up with $K(f) = 1/d(a)+1/d(b)+1/d(c) -3/2 + 1 
= 1/d(a)+1/d(b)+1/d(c)-1/2$. 
For 2-dimensional manifolds with boundary, the edge curvature at the rim is 
$K(v) = 1/d(a) + 1/d(b) - 2/3$
and the face curvature is at the boundary $K(f) = 1/d(a) + 1/d(b) + 1/d(c) -2/3$. 
Beside the familiar Gauss Bonnet $\chi(G)=\sum_{v \in V} K(v)$
for vertex curvatures, there are now new Gauss-Bonnet statements for edge or face curvatures
$\chi(G) = \sum_{e \in E} K(e)$ and $\chi(G)=\sum_{f \in F} K(f)$. 

\paragraph{}
To summarize: in the special case of 2 dimensional manifolds, (graphs for which every unit
sphere $S(v)$ is a circular graph with $d(v) \geq 4$ elements), we have:

\begin{center} 
\begin{tabular}{|l|l|l|}  \hline
       &  in the interior                    & 
          at the boundary  \\ \hline
vertex &  $1-\frac{d(v)}{6}$                             & 
          $\frac{2}{3}-\frac{d(v)}{6}$ \\ \hline
edge   &  $\frac{1}{d(a)} + \frac{1}{d(b)} -\frac{1}{3}$ & 
          $\frac{1}{d(a)} + \frac{1}{d(b)} -\frac{2}{3}$ \\ \hline
face   &  $\frac{1}{d(a)} + \frac{1}{d(b)} +\frac{1}{d(c)} -\frac{1}{2}$ &  
          $\frac{1}{d(a)} + \frac{1}{d(b)} +\frac{1}{d(c)} -\frac{2}{3}$ \\ \hline
\end{tabular}
\end{center}

\begin{figure}[!htpb]
\parbox{13cm}{
\parbox{6cm}{ \scalebox{0.08}{\includegraphics{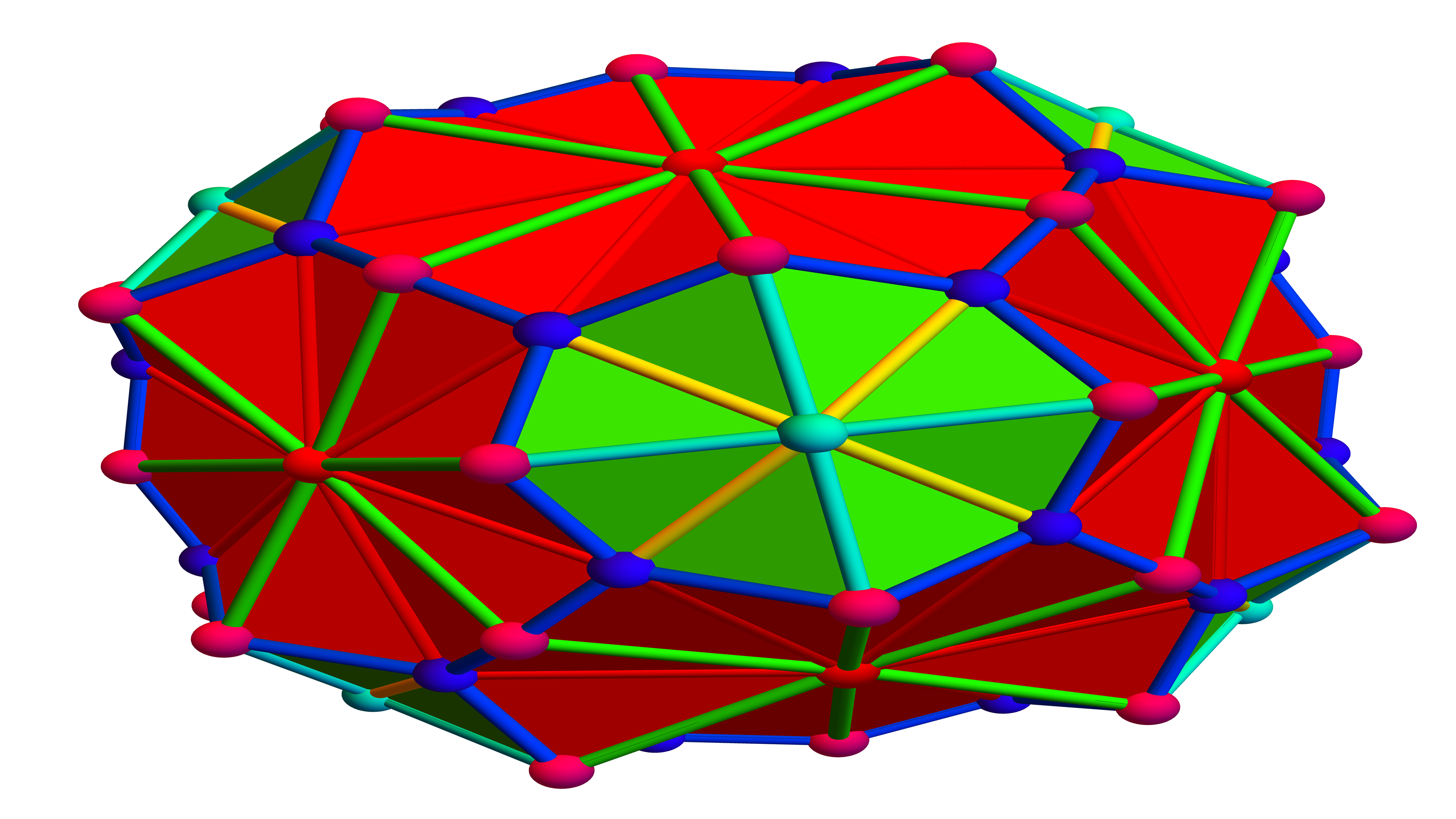}}}
\parbox{6cm}{\scalebox{0.08}{\includegraphics{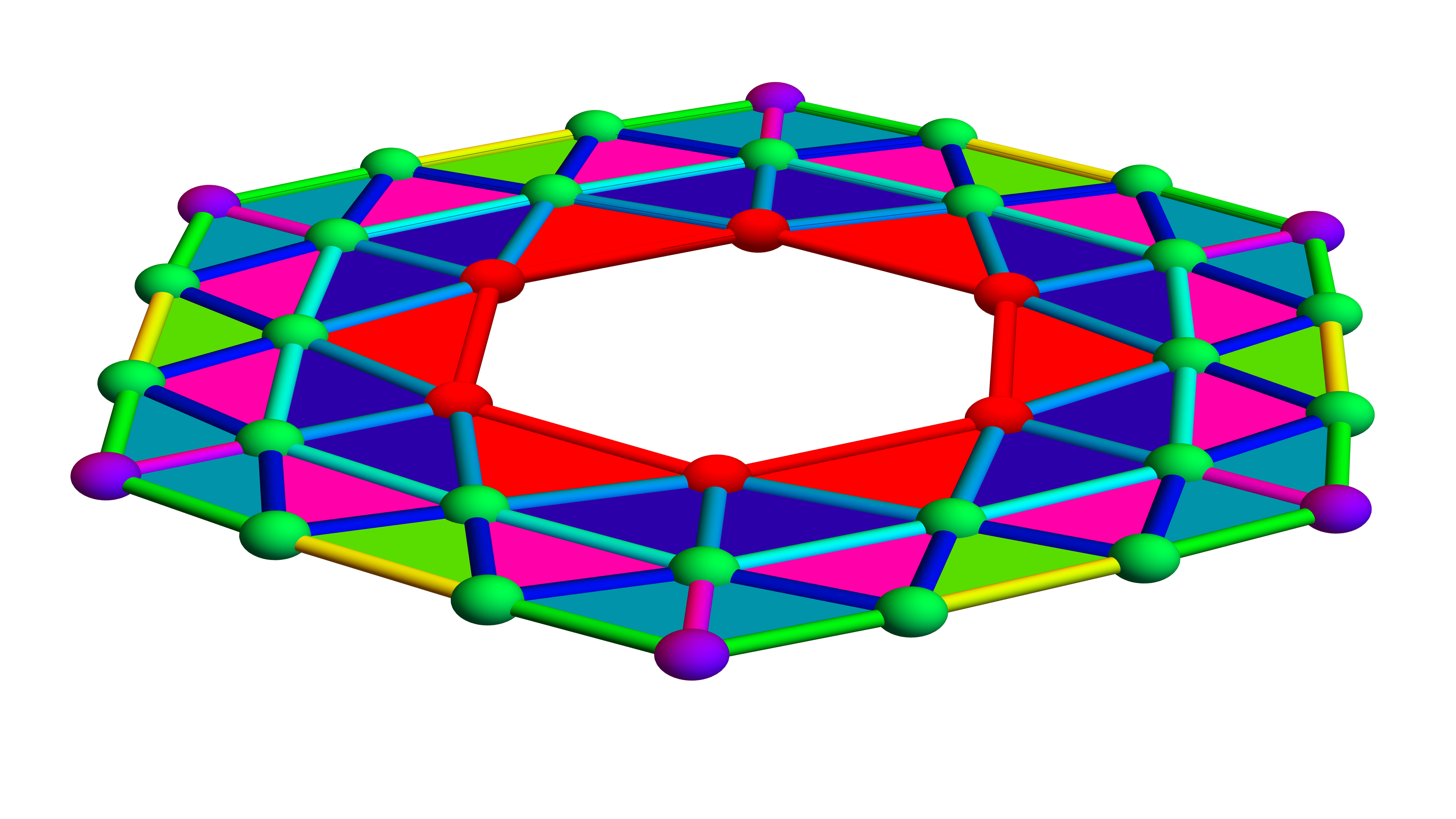}}}
}
\label{cube}
\caption{
The manifold without boundary is the Barycentric refinement of a stellated cube.
To the right, we see an example of a manifold with boundary that is flat in
the interior. Also the edge curvature of interior edges is zero.
}
\end{figure}

\section{The general case}

\paragraph{}
We assume from now on that $G$ is a {\bf finite abstract simplicial complex} meaning
a finite set of non-empty sets closed under the operation of taking non-empty subsets.
In order to introduce curvatures that are $k$-forms, meaning functions on 
$G_k = \{ x \in G, {\rm dim}(x) = |x|-1= k \}$, the 
$k$-dimensional simplices, we need to assume that the finite abstract simplicial 
complex $G=\bigcup_{k=0}^q G_k$ has the property that 
$\bigcup_{x \in G_k} x= V$ with $V=\bigcup_{x \in G} x$.
We say in that case that $\bigcup_{x \in G_k} x$ {\bf covers $G$}. A {\bf pure q-dimensional 
complex} for example, is a complex $\bigcup_{k=0}^q G_k$ for which the {\bf facets} 
$G_q$ cover $G$ meaning that every $x \in G$ is a subset of some facet. 
For $G=\{ \{1\},\{2\},\{3\},\{1,2\} \}$ for example, $G_0$ covers $G$ but $G_1$ does
not cover $G$. The statement {\bf $G_1$ covers $G$} is equivalent to {\bf $G$ is connected}.

\paragraph{}
We define then the {\bf k-form-curvature} for $0 \leq k \leq q$ as
$$  K_k(x) = \sum_{y \subset x} \frac{\omega(y)}{d_k(y)} + \sum_{j>k} (-1)^j \frac{d_j(x)}{\Binomial{j+1}{k+1}} $$
with $d_k(x) = |\{ y \in G_k, y \subset x\}$.
We get so functions $K_k:G_k \to \mathbb{R}$ which can be called a $k$-form,
\footnote{As defined the curvature is a function on $x$, where $x$ is seen as a set. It 
can be extended to ordered sets by setting $K(\pi(x))={\rm sign}(\pi) K(x)$ so
that it becomes a discrete differential form. This does not affect the 
integral $\sum_{x \in G_k} K_k(x)$ of the differential form. This is the point of view
of Whitney who both worked in the continuum and discrete.}
as it is a function on $G_k$, or elements in $l^2(G_k)$, if we want to use an inner product
on such functions. \footnote{$K_k(v)$, the curvature has
nothing to do with the complete graph $K_k$.}

\paragraph{}
In the case $k=0$, the first part $\sum_{y \subset x} \frac{\omega(y)}{d_k(y)}$ is zero. 
We get then the familiar {\bf Levitt curvature}, where $d_j(x)$ counts the number of $j$ dimensional 
simplices attached to the vertex $x$. 
$$  K_0(x) = \sum_{j} (-1)^j \frac{d_j(x)}{j+1}  \; . $$
An other special case is the {\bf facet curvature} for $k=q$. It assumes that the simplical complex is
covered by the highest $q$-dimensional facets. One calls such a simplicial complex {\bf pure}
or {\bf homogeneous} of maximal dimension $q$. Under this assumption, the second part of the form curvature
is zero and we end up with 
$$  K_q(x) = \sum_{y \subset x} \frac{\omega(y)}{d_q(y)} \; , $$
where the {\bf facet degree} $d_q(y)$ of a simplex $y$ counts the number of $q$ dimensional simplices 
which contain $y$. 

\paragraph{}
We have for every $k \geq 0$ the Gauss-Bonnet formula:

\begin{thm}[Gauss-Bonnet] If $\bigcup_{x \in G_k} x$ covers $G$, then
$\sum_{x \in G_k} K_k(x)=\chi(G)$. \end{thm}

\begin{proof}
The proof is almost tautological if one thinks in terms of energies $\omega(y)$ attached to 
the simplices $y \in G$. Since $G_k$ covers $G$, the distribution of energies does take care
of all energies. There are two type of transactions. Either $y \subset x$, then the energy $\omega(y)$
is split into $d_k(y)$ simplices of dimension $k$. Or then the later case $x \subset y$, where we can
count how many times this happens. We have $\Binomial{j+1}{k+1}$ simplices of dimension $k$ in a simplex
of dimension $j$. 
\end{proof} 

\paragraph{}
{\bf Examples:} We have already seen the cases for $2$-manifolds, where $\chi(G)=|V|-|E|+|F|=|G_0|-|G_1|+|G_2$, it becomes
$K_0(a) = 1-d(a)/6$, $K( (a,b) )=1/d(a)+1/d(b)-1/3$ and $K( (a,b,c) ) = 1/d(a)+1/d(b)+1/d(c)-1/2$.
For $3$-manifolds, it simplifies to $K_0(a)=1-d(a)/2+d_2(a)/3-d_3(a)/4$,
$K_1((a,b)) = 1/d_1(a)+1/d_1(b)-1+d_2(a,b)/12$ and
is $K_2((a,b,c)) = 1/d_2(a) +1/d_2(b)+1/d_2(c)-1/d_2(a,b)-1/d_2(b,c)-1/d_2(a,c)+1/2$
is $K_3((a,b,c,d)) = \sum_{y \subset \{a,b,c,d\}} \omega(y)/d_3(y)$.  We will come back to this in the 
3-manifold section. 

\paragraph{}
If $G$ has $n$ elements, there is a $n \times n$ matrix $D = d+d^*$ which encodes the
exterior derivatives of the complex. 
\footnote{We called it the Dirac operator \cite{KnillILAS} of the graph at a time, when 
most graph theorists would see "graphs" as {\bf one-dimensional simplicial complexes}.}
The square $L=\oplus_{k=1}^q L_k$ is a block diagonal matrix and the kernel of
$L_k: l^2(G_k) \to l^2(G_k)$ is the space of {\bf harmonic k-forms}. 
The numbers $b_k = {\rm dim}({\rm ker}(L_k)$ are called
the {\bf Betti numbers}.  The {\bf Euler characteristic} of $G$ is $\sum_{x \in G} \omega(x)$ with 
$\omega(x) = (-1)^{{\rm dim}(x)}$. If we define the {\bf super trace} for an operator on $l^2(G)$ as
${\rm str}(A) = \sum_{k}  A_{kk} \omega_k$, then ${\rm str}(1)= \chi(G)$. The McKean-Singer symmetry
\cite{McKeanSinger,knillmckeansinger}
tells ${\rm str}(L^k)=0$ for $k \geq 0$ so that ${\rm str}(e^{-t L})= \chi(G)$. 
Since for $t \to \infty$ this gives $\sum_{k=0}^d (-1)^k b_k = \chi(G)$, 
the Euler-Poincar\'e formula follows. 

\section{3-manifolds}

\paragraph{}
A finite simple graph is a {\bf 3-manifold} if every unit sphere $S(v)$ is a 
{\bf 2-sphere}, meaning a connected 2-dimensional manifold of Euler 
characteristic $2$. A {\bf 2-manifold} is a finite simple graph for which every 
unit sphere is a cyclic graph of length $4$ or more. 
2-spheres can also be characterized as
2-manifolds $G$ which after removing a vertex become contractible, a recursive notion
which means that there is a vertex $v$ such that both $G-v$ and $S(v)$ are contractible.
\footnote{See \cite{LusternikSchnirelmann2024} for a recent more expository document}.
3-manifolds can be given as graphs. Libraries like \cite{Manifoldpage} give 
them by listing the facets $G_3$, the highest dimensional simplices. The 
simplicial complex $G$ generated by them then defines the graph with vertices
$V=G_0$ and edges $E=G_1$. 

\paragraph{}
In dimension $3$, the curvature $K(v)=K_0(v)$ is given by 
$$  K(v)=1-\frac{V(v)}{2}+\frac{E(v)}{3} - \frac{F(v)}{4} \; , $$
where $V(v),E(v),F(v)$ are the number of vertices, edges and faces in the
unit sphere $S(v)$ of a vertex $v$. From the Euler relation $V-E+F=2$  
and the {\bf Dehn-Sommerville relation} $2E=3F$, (2 tetrahedra share exactly 
one triangle), we easily get $K(v)=0$. 
In \cite{cherngaussbonnet}, we used the Dehn-Sommerville relations to get 
simplified relations in higher dimensions like $K(v)= 1-E(v)/6+F(v)/10$ for 
$4$-manifolds and $K(v)=-E(v)/6+F(v)/4-C(v)/6$ for $5$-manifolds, where 
$C(v)$ is the number of $3$-dimensional {\bf chambers} in the unit sphere $S(v)$. 
We had at that time not been aware of the second Dehn-Sommerville relation 
for 5-manifolds, which is equivalent to $K(v)=0$ for $5$-manifolds.

\paragraph{}
We looked in \cite{cherngaussbonnet} already at various examples like the 16-cell
$K_{2,2,2,2}$ which is the suspension of the octahedron $K_{2,2,2}$.
The 600-cell is a 3-sphere for which all unit spheres are icosahedra.
Or then an octahedral tesselation of the 3-torus.  All the 5-manifold examples
we have seen then displayed $K(v)=-E(v)/6+F(v)/4-C(v)/6=0$. The solution is 
given most elegantly in \cite{dehnsommervillegaussbonnet}: the curvature 
is a Dehn-Sommerville invariant for the unit sphere. Dehn-Sommerville graphs
are graphs for which the Euler characteristic matches the sphere case and for
which all unit spheres are Dehn-Sommerville. It is a very natural class of 
graphs which contains d-spheres but does not refer to contractibility.

\paragraph{}
The edge curvature of a 3-manifold is for an edge $x=(a,b)$ given by 
$$  K_1(x) = \frac{1}{d(a)} + \frac{1}{d(b)} + \sum_{j=1,2,3} (-1)^j \frac{d_j(x)}{\Binomial{j+1}{2}}  $$
which is $K_1(x) = \frac{1}{d(a)} + \frac{1}{d(b)} -1 + \frac{d_2(x)}{3} - \frac{d_3(x)}{6}$,
where $d_2(x) = |S(a) \cap S(b)|$ is the number of triangles containing $(a,b)$
and $d_3(x) = |S(a) \cap S(b)|$ is the number of chambers (=tetrahedra) containing $(a,b)$. So, 
for a 3-manifold, we have 

\begin{propo}[Edge curvature for a 3-manifold]
For a 3-manifold, we have the edge curvature
$$K_1( (a,b) ) =  \frac{1}{d(a)} + \frac{1}{d(b)} + \frac{d(a,b)}{6} - 1 $$
where $d(a),d(b)$ are vertex degrees and $d(a,b)$ is the edge degree.
\end{propo}

\begin{figure}[!htpb]
\scalebox{0.55}{\includegraphics{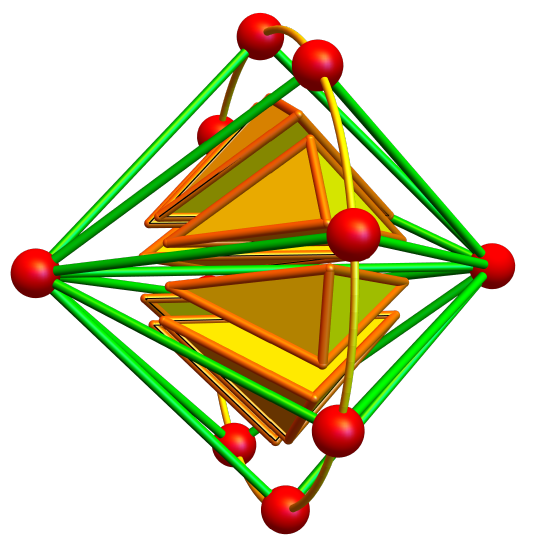}}
\label{edge degree}
\caption{
The edge degree of an edge $x=(a,b)$ in a 3-manifold is the number of 
tetrahedra hinging at $(a,b)$. We have now vertex degrees $d(a)=|S(a)|$
and $d(b) =|S(b)|$ and an edge degree $d(a,b)  = |S(a) \cap S(b)|$. 
}
\end{figure}

\paragraph{}
Examples: \\
a) For the smallest 3-manifold $K_{2,2,2,2}$, which is also known as the {\bf 16-cell},
with $f$-vector $(8,24,32,16)$, we have $d(a)=d(b)=6$ and $d(a,b)=4$ for all edges 
$(a,b)$. The edge curvature is $1/6+1/6+4/6-1$ which is constant $0$.  \\
b) For the 600-cell, also one of the Platonic 4-polytopes (3-spheres embedded 
in 4 dimensions), the $f$-vector is $(120, 720, 1200, 600)$. Every unit sphere $S(v)$ is 
an icosahedron with $|S(v)|=12$ so that $d(v)$ is constant $12$. The edge degree of 
one of the 720 edges is $|S(a) \cap S(b)| = 5$. Now 
$\frac{1}{12} + \frac{1}{12} + \frac{5}{6} - 1 = 0$. 

\paragraph{}
The {\bf face curvature} of a face $x=(a,b,c)$ in a 3-manifold is given by 
$$   K_2(x) =  \frac{1}{d_2(a)} + \frac{1}{d_2(b)} +  \frac{1}{d_2(c)} 
             - \frac{1}{d_2(a,b)} - \frac{1}{d_2(b,c)} - \frac{1}{d_2(a,c)} + 1-2/4  $$
where $d_2(a)$ is the number of triangles which contain $a$ and $d_2(a,b)$ the number of
triangles which contain $(a,b)$. We called this the edge degree $d(a,b)$ before. 
The last part is because there are exactly two tetrahedra attached to each face and each tetrahedron
shares with 4 faces.

\begin{propo}[Face curvature for a 3-manifold]
For a 3-manifold, we have
$$K_2( (a,b,c) ) =   \frac{1}{d_2(a)} + \frac{1}{d_2(b)} +  \frac{1}{d_2(c)} 
                   - \frac{1}{d_2(a,b)} - \frac{1}{d_2(b,c)} - \frac{1}{d_2(a,c)} +\frac{1}{2}  \; , $$
where $d_2(a)$ counts the number of triangles containing $a$ and $d_2(a,b)=d(a,b)$ is the edge degree.
\end{propo}

\paragraph{}
If we go back to the two Platonic examples. \\
a) For the 16 cell we have exactly 12 triangles that contain a vertex (as the unit sphere
is an octahedron with 12 edges). We have exactly 4 triangles that contain an edge $(a,b)$. 
This is the edge degree $|S(a) \cap S(b)|$. The face curvature is now
$K_2((a,b,c)) = 1/12+1/12+1/12-1/4-1/4-1/4+1/2 =0$. The face curvature is constant $0$. \\
b) For the 600 cell, where the unit sphere is an icosahedron we have
$d_2(a) = 30$ as there are 30 edges in an icosahedron. We also have $d_2(a,b) = 5$. Therefore
$K_2((a,b,c)) = 1/30+1/30 +1/30-1/5-1/5-1/5+1/2 = 0$. \\ 

\paragraph{}
Lets look at the facet curvature
$$  K_3(x) = \sum_{y \subset x} \frac{\omega(y)}{d_q(y)} \; , $$
of a 3-manifold. We could also call it the {\bf chamber curvature} of a chamber $x=(a,b,c,d)$. 
We can split it up into 
$$ \sum_{y \in \{(a),(b),(c),(d)\}} \frac{1}{d_3(y)}
 - \sum_{y \in \{(a,b),(a,c),(a,d),(b,c),(b,d),(c,d)} \frac{1}{d_3(y)} 
 + \sum_{y \in \{(b,c,d),(a,c,d),(a,b,d),(b,c,d)} \frac{1}{d_3(y)} -1 $$
where $d_3(y)$ is the number of tetrahedra shared with the simplex $y$. 
For a vertex $v$, we have $d_3(v)=F(v)$, the number of faces in $S(v)$.
For an edge $e=(a,b)$ we have $d_3((a,b)) =d(a,b)$ the edge degree
and for a face $f=(a,b,c)$ we have $d_3(a,b,c)=2$. This can be written as follows:

\begin{propo}[Facet curvature for a 3-manifold]
For a 3-manifold, we have for a chamber $x=(a,b,c,d)$
$$K_3(x) =  \frac{1}{F(a)} + \frac{1}{F(b)} +  \frac{1}{F(c)} + \frac{1}{F(d)} 
                   - \frac{1}{d(a,b)} - \frac{1}{d(a,c)} - \frac{1}{d(a,d)} - \frac{1}{d(b,c)} - \frac{1}{d(b,d)} - \frac{1}{d(c,d)} 
                   + 1 \; , $$
where $d_3(a)=F(a)$ is the face degree of a vertex and $d(a,b)$ is the degree of an edge. 
\end{propo}

\paragraph{}
Again, let us look at the Platonic examples. \\
a) For the 16 cell we have $F(v)=8$ for all vertices $v$ and 
$|S(a) \cap S(b)|=4$ for all edges $e=(a,b)$.  Then, there are 
Therefore $K_3((a,b,c,d)) = 4/8-6/4 + 1 =0$.  The face curvature is also constant $0$. \\
b) For the 600 cell, where the unit sphere is an icosahedron we have
$F(a) = 20$ as there are 20 triangles in an icosahedron. We also have $d_3(a,b)=d(a,b) = 5$. 
Therefore $K_3((a,b,c,d)) = 4/20-6/5+1 = 0$. \\ 

\begin{figure}[!htpb]
\scalebox{0.08}{\includegraphics{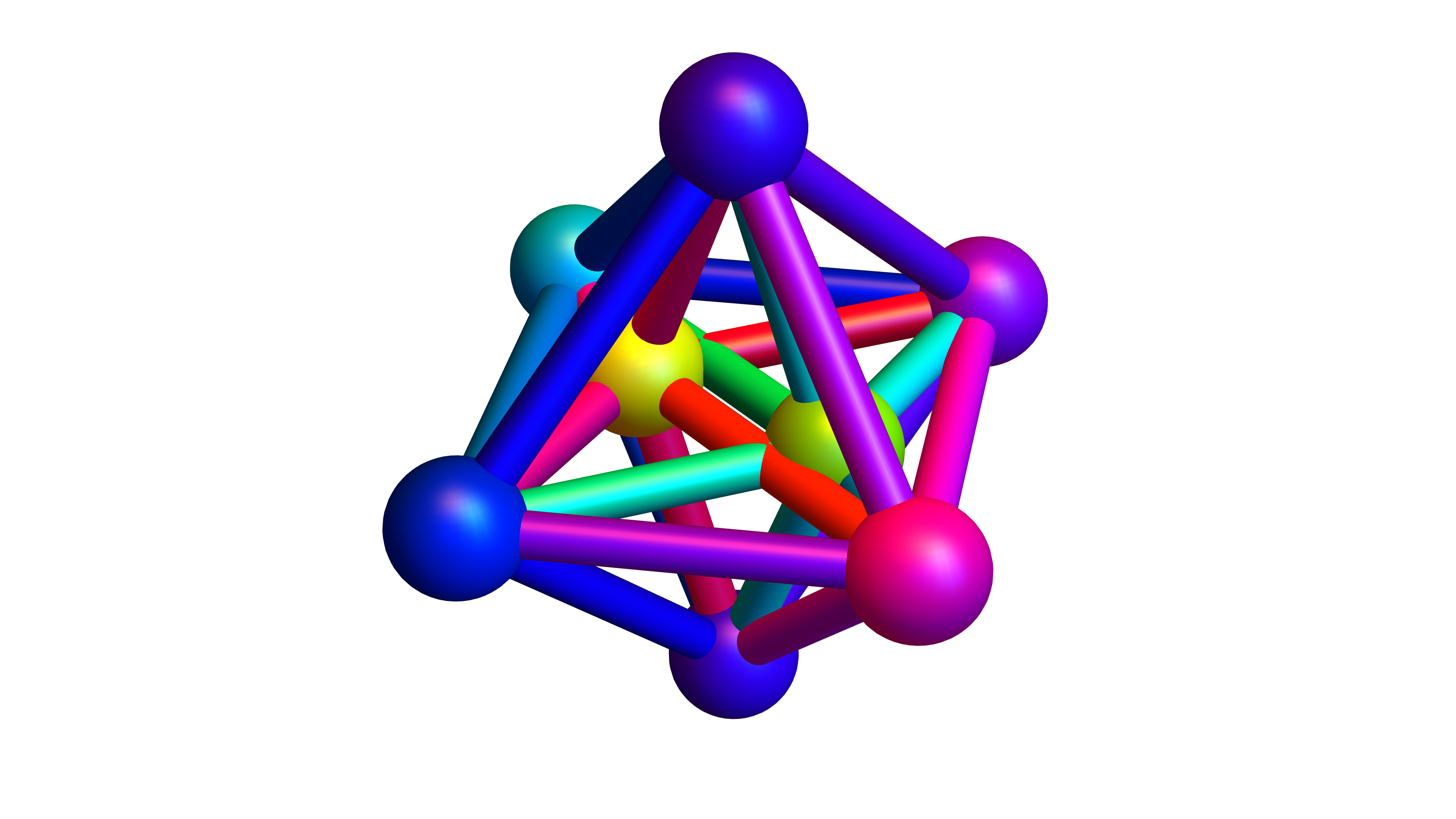}}
\scalebox{0.08}{\includegraphics{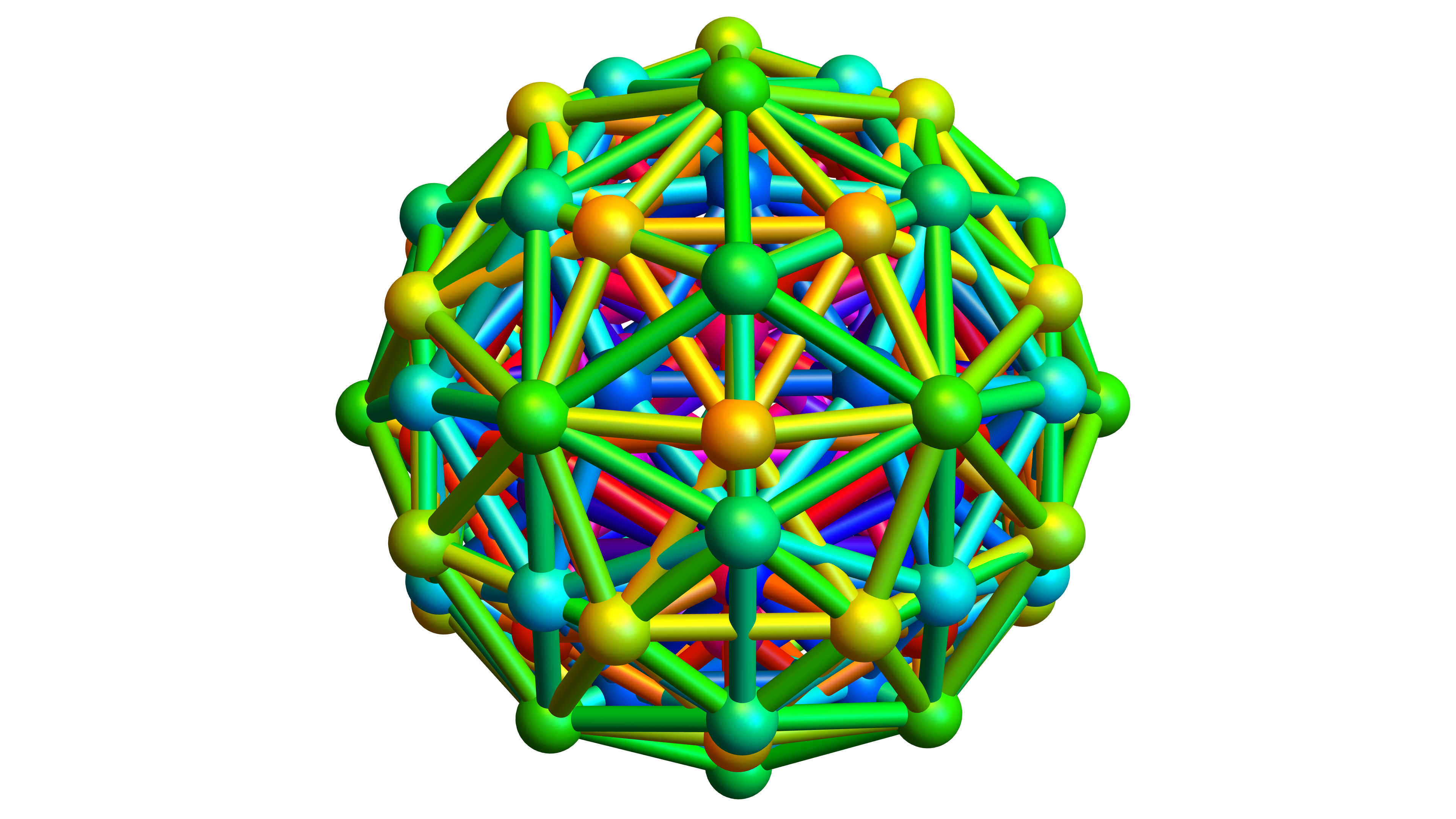}}
\label{Platonic}
\caption{
The 16 cell $K_{2,2,2,2}$ and the 600 cells are the only Platonic
3-spheres.  All the curvatures $K_k(x)$ are constant zero everywhere. 
}
\end{figure}

\paragraph{}
We should note that the other four Platonic 3-spheres (regular 4-polytopes)
are not triangulations of a 3-sphere. The unit spheres $S(v)$ of a Platonic 
d-manifold have to be Platonic (d-1) manifold and there are only 2 Platonic 2-spheres 
(assuming to be a 2-manifold), the octahedron $K_{2,2,2}$ and the 
icosahedron. In all other Platonic cases, the unit spheres are not spheres.
Gauss-Bonnet quickly shows that the only Platonic $d$-spheres for $d>3$ are
the cross polytopes $K_{2,2,\dots, 2}$. Gauss-Bonnet prevents that a 600 cell
is the unit sphere of all vertices of a $4$-sphere:
the curvature of a 4-manifold is $K(v)= 1-E(v)/6+F(v)/10$.  Because
The f-vector of the 600 cell is $(120, 720, 1200, 60)$, the curvature
would have to be constant $K(v)=1-720/6+1200/10=1$. 
The Euler characteristic of a 4-sphere is by Euler's Gem formula equal to $2$. 
There is obviously no 4-manifold with 2 vertices. 
See \cite{AmazingWorld} for an exposition.

\section{Deformation of the curvature}

\paragraph{}
The curvature produces a dynamics on $k$-forms. We are using the {\bf wave operator} 
$w_t=\cos(Dt)$ which solves the wave equation $w''=-L w$. 
\footnote{More generally, we could take $w=e^{iDt}$ or $e^{-iDt}$ and so get
complex-valued curvatures.}
Define $\omega_t(x)=(-1)^{{\rm dim}(x)} \cos(D t)(x,x)$ and the time-dependent
curvature
$$   K_{k,t}(x)= \sum_{y \leq x} \frac{\omega_t(y)}{{\rm deg}_k(y)} 
   + \sum_{y \geq x} \frac{\omega_t(y)}{\Binomial{{\rm dim}(y)+1}{k+1}}  \; . $$
This function is implemented in the Mathematica code below. If $t=0$ it is the
form curvature $K_k(x)$. For every $t$ and every $k$, we still have Gauss Bonnet: 

\begin{coro}[Gauss-Bonnet for time dependent curvatures]
If $G_k$ covers $G$, then $\chi(G) = \sum_{x \in G_k} K_{k,t}(x)$ 
for all $t$.
\end{coro}

\begin{proof} 
For any function $f$, the functional calculus defines $f(D)$ as $D$ is symmetric. 
The wave operator $\cos(D t)$ is a symmetric $n \times n$ matrix. 
The McKean-Singer symmetry assuring that ${\rm str}(L^k)=0$ for all $k>0$ implies that
the super trace of $\cos(D t)$ is equal to $\chi(G)$ and so that it is independent of $t$. 
We can now use the diagonal values $\omega_t(x)$ instead of $\omega(x)=\omega_0(t)$. 
\end{proof} 

\paragraph{}
Note that the curvature functions $K_{k,t}$ on $G_k$ do not satisfy 
the wave equation by themselves. It is that the underlying energies values 
given by the diagonal elements of the wave operator that move according to the 
wave equation. The curvature $K(x)$ on k-forms then aggregates all energies of 
simplices close to $x$.  There is no choice of initial condition for this 
evolution because $K_{k,0}(x)=K_k(x)$ is the curvature on $k$-forms, 
introduced before. 

\begin{figure}[!htpb]
\scalebox{0.28}{\includegraphics{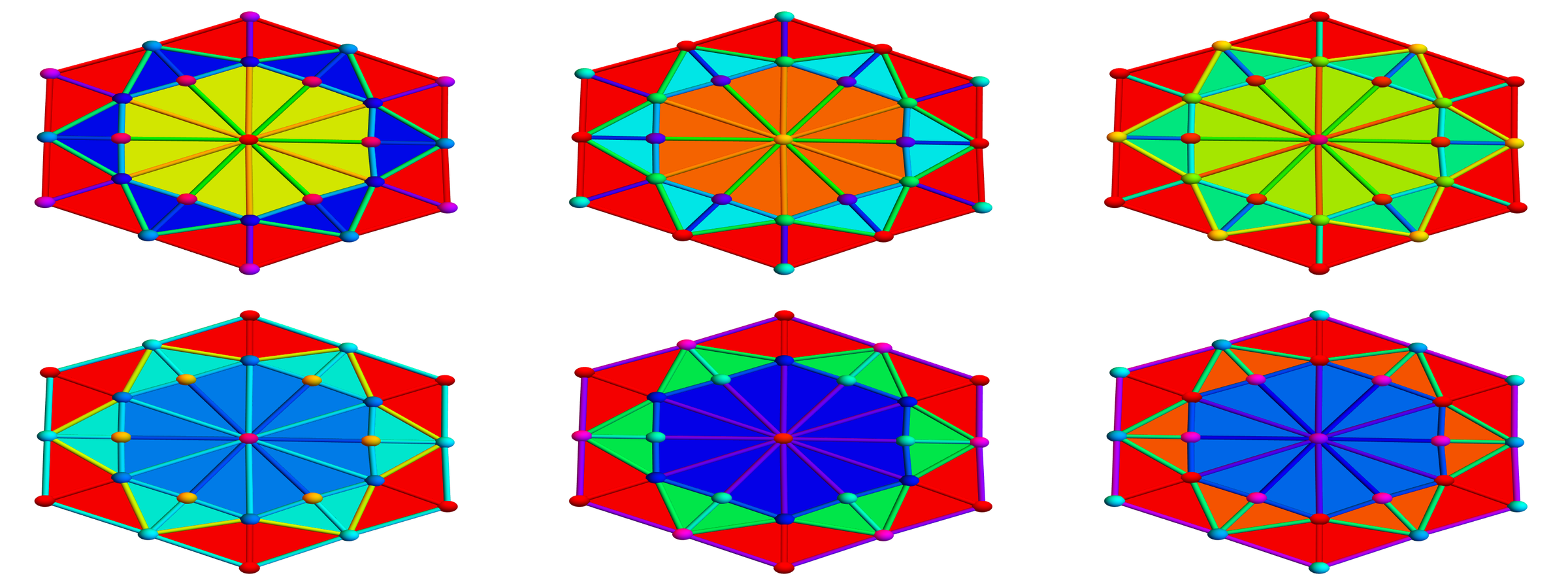}}
\label{cube}
\caption{
A wave deformation in the case of second Barycentric refinement of $K_3$. 
The wave is evaluated at times $k/2$ for $k=1,\dots, 6$.
The colors encode the vertex, edge and wave curvatures. 
}
\end{figure}

\paragraph{}
{\bf Non-linear wave versions} come from 
{\bf integrable Lax pair systems} $D'=[g(D)^+-g(D)^-,D]$ \cite{IsospectralDirac,IsospectralDirac2} , 
leading to a unitary $Q_t'=(g(D)^+-g(D)^-)Q_t$ satisfying $D_t=Q_t^* D_0 Q_t$. 
As noted in \cite{QRKnill}, the {\bf non-linear} {\bf iso-spectral deformation}
$D'=[g(D)^+-g(D)^-,D]$ is equivalent to the QR deformation like $e^{-t g(D_0)} = Q_t R_t$
and get a new Dirac matrix $D_t=Q_t^* D Q_t$ which has the form 
$D_t=d_t + d_t^* = Q_t^* d Q_t + Q_t^* d^* Q_t$ which is now 
block tri-diagonal $D_t = c_t + c_t^* + m_t$. The deformation can move spontaneously 
also in the complex by modifying to 
\begin{equation}
\label{complex evolution}
D'=[g(D)^+-g(D)^- + i c g(D)^0,D] 
\end{equation}
with some nonzero constant $c$.

\paragraph{}
If we want a linear and discrete time evolution, we could look at 
$T(u,v) = (c 2 Du-v, u)$, where $c$ is so small that $||cD||<1$ rather than the 
pair $u'= \pm i c D u$ or wave equation \cite{Kni98}. 
There is then still a unitary evolution $U_t = \exp(i t \arccos(cD))$, and 
the super trace of $U_t$ is still constant. The advantage of such a discrete 
time evolution is that it has {\bf finite propagation speed}, meaning causality. 
Causality can also be obtained in the nonlinear Lax deformation mentioned above.
We just need to make sure that $e^{-t g(D)}$ is local which is achieved for
$g(D)=-\log(1+c D)$ which is defined if $||cD||<1$. In that case 
$e^{-t g(D)}=(1+cD)^t$. The orthogonal matrix $Q_t$ then does not relate points
in distance larger than $t$. 

\paragraph{}
By the way, we introduced the discrete time evolution in \cite{Kni98} 
for numerical purposes, like to measure the nature of spectra of discrete Schr\"odinger
operators. The growth rate of the spectral measures determines the spectral type. 
If we can numerically compute the growth rate with an evolution that has finite propagation 
speed, we know that any boundary effects do not play a role as long as the wave has not
reached the boundary. When using rational or algebraic numbers, the evolution is {\bf exact}.
If the target space (usually $\mathbb{R}$ is replaced by a ring element, we can study
arithmetic versions. With a finite ring, we get cellular automata \cite{Wolfram86}.
Also in the present geometric frame work, we could replace real or complex valued functions
with functions taking values in other rings, like also finite rings. An obvious generalization
is to expand the non-linear evolution to quaternions. We would just have to replace the $i$ in the
differential equation (\ref{complex evolution}) with a unit quaternion. 

\section{Poincar\'e-Hopf}

\paragraph{}
A scalar function $g:G_k \to \mathbb{R}$ is called {\bf locally injective} if 
$g(x) \neq g(y)$ if either $x \cap y$ is not empty or if $x \cup y$ are in a common simplex. 
Such a function defines a {\bf total order} on the vertices of each simplex $x$ of $G_k$. 
It allows to every simplex $x \in G$ assign a specific $y \in G_k$ with 
$y \subset x$ or $x \subset y$ on which $g$ is the maximum. More generally, we just 
assign a function $F: G \to G_k$ such that $F(y) \subset y$ or $y \subset F(y)$. This gives the 
k-dimensional {\bf Poincar\'e-Hopf index}
$$ i_{g,k,t}(x)= 1- \chi_t(S_g^+) -\chi_t(S_g^-) \; , $$
where $S_g^+ = \{y, y \subset x, g(y)<g(x) \}$
and $S_g^- = \{y, x \subset y, g(y)<g(x) \}$. 

\paragraph{}
There is now a more general Poincar\'e-Hopf formula, where the integer indices are located on 
$k$-dimensional simplices. It again holds only if the $k$-dimensional part of space covers it
meaning $\bigcup_{x \in G_k} x=\bigcup_{x \in G} x = V$.

\begin{thm}[Poincar\'e-Hopf]
If $\bigcup_{x \in G_k} x$ covers $G$, then $\chi(G) = \sum_{x \in G_k} i_{k}(x)$. 
\end{thm}
\begin{proof}
Every simplex $y \in G$ will be pointed to a specific $x \in G_k$. But this needs that 
the $k$-simplices cover $G$.  Moving the energies
accordingly from $G$ to $G_k$ does not change the total energy which is the Euler characteristic. 
\end{proof} 

\section{Index expectation}

\paragraph{}
For $k=0$, where $K_0$ is the Levitt curvature on vertices, the curvature is constant $0$ if $G$ is 
an odd-dimensional manifold. Conjectured in \cite{cherngaussbonnet}, it was
verified using index expectation in \cite{indexexpectation,indexformula} and later seen it as
a Dehn-Sommerville relation \cite{DehnSommerville}, establishing so zero curvature also 
for odd-dimensional Dehn-Sommerville manifolds \cite{DehnSommerville}, 
simplicial complexes that generalize $d$-spheres but satisfy Dehn-Sommerville relations. 
Graphs for which the unit spheres are Dehn-Sommerville manifolds of dimension 
$d-1$ can be seen as generalized $d$-manifolds. We wondered whether for odd dimensional 
manifolds, also higher curvatures $K_k$ are constant $0$. It turns out that this is not the
case in general. 

\paragraph{}
The Poincar\'e-Hopf story on k-dimensional simplices is a tale about divisors on k-dimensional
parts of space. One can see them as integer-valued curvatures. If we take a probability measure
on the space of locally injective functions on $G_k$ and take the expectation of
$$ \sum_{x \in G_k} i_f(x) = \chi(G) \; .  $$
We get on the left hand side sum of curvatures
$$  K(x)={\rm E}[i_f(x)]  $$
on $G_k$. Index expectation links the discrete and continuum.
if $G$ is a finite simple graph with Euler characteristic $\chi(G)=\sum_k (-1)^k f_k(G)$, 
where $f_k(G)$ counts the number of k-dimensional complete sub-graphs in $G$, the Levitt curvature
$K(v) = 1-\sum_{k} (-1)^k f_k(S(v))/(k+1)$ satisfies Gauss-Bonnet $\sum_{v \in V} K(v) = \chi(G)$.
Allowing the probability measure to vary gives more flexibility. We can for example embed the network 
$G$ into a high dimensional linear space $E$ and the space of linear functions on $E$ (parametrized by
vectors on the unit sphere equipped with the homogeneous measure) as probability space, then 
${\rm E}[i_f(x)]$ is a curvature which is close to the continuum if the network 
is close to the physical space we want to model. 
If we take random IID values $f(x)$ on $G_k$ with uniform distribution in some interval
for example, then the index expectation is the 
curvature $K_k(x)$ considered above. 

\paragraph{} 
For $K_0$, there is a generalization of Gauss-Bonnet in which one 
does not work with Euler characteristic but more generally with the simplex generating
function 
$$   f_G(t)=\sum_{k=0}^{\infty} f_k(G) t^k $$ 
which for $t=-1$ relates to the Euler characteristic by $\chi(G)=1-f_G(-1)$. 
The Levitt Gauss-Bonnet formula now reads more naturally as
$$   f_G'(t) = \sum_{v \in V} f_{S(v)}(t)  $$ 
which after integration and evaluating at $t=-1$ gives $1-\chi(G)$ and so Gauss-Bonnet. 
Since the Levitt formula is an expectation of Poincar\'e-Hopf indices, this Levitt curvature is 
the discrete analog of the Gauss-Bonnet-Chern curvature $K$ which involves the Riemann
curvature tensor $R_{i,j,k,l}$. It is 
$$ K = (-1)^n \sum_{i,j} \epsilon_{i_1,...,i_{2n}} \epsilon_{j_1,...,j_{2n}} 
         R_{i_1 i_2 j_1 j_2}  \cdots R_{i_{2n-1}} i_{2n} j_{2n-1} j_{2n} $$
for Riemannian manifolds. In the discrete, curvature is zero for odd-dimensional 
manifolds, in the continuum, the intrinsic curvature of odd-dimensional manifolds is not even
considered. 

\begin{figure}[!htpb]
\scalebox{0.75}{\includegraphics{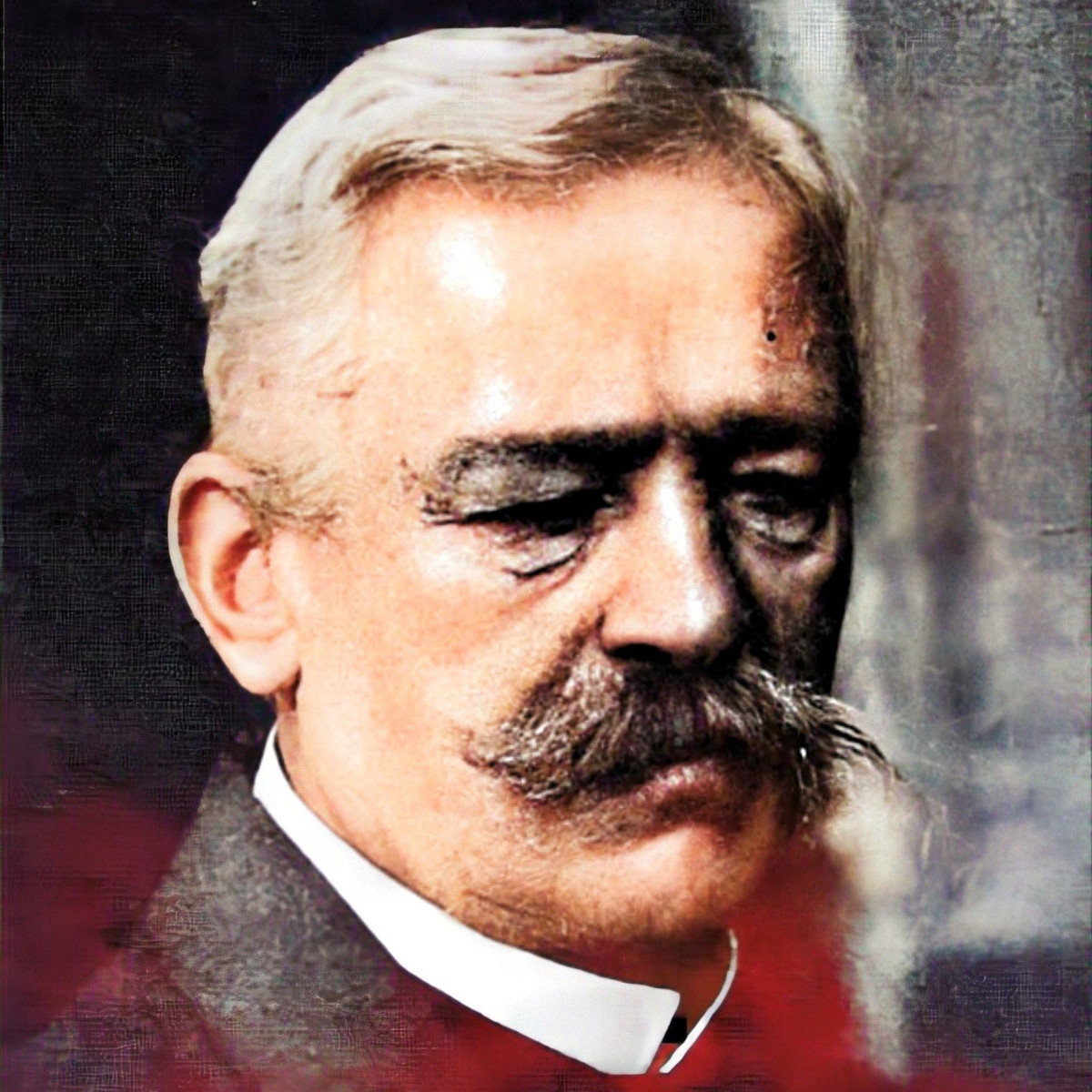}}
\label{eberhard}
\caption{
  Victor Eberhard (1861-1927) in 1922. Eberhard lost his eye sight as a 13 year old and showed
  enormous resilience, developing strong intuition, memory, humor and positivity to overcome 
  the handicap. The photo source is the biography \cite{Rosenthal} about Eberhard. 
   We colorized the photograph electronically using a tool provided by MyHeritage.com.
}
\end{figure}

\section{Puiseux relation}

\paragraph{}
In \cite{elemente11}, we considered the Puiseux type curvature $K(v)=2 |S(v)|-|S_2(v)|$
for flat $2$-dimensional plains, where we proved $\sum_{p \in \delta G}  K(p) = 12 \chi(G)$. 
We had felt at that time that a good curvature should be a second order notion, meaning
that it should involve neighborhoods of distance $2$. This was motivated by the
classical Puiseux formula $K(v) = \lim_{r \to 0} 3 \frac{2\pi r-|S_r|}{\pi r^3}$ which is
equivalent to the R2-D2 formula ({\bf R}adius times two in {\bf D}imension 2) 
$$   K(v) = \lim_{r \to 0} \frac{2|S_r|-|S_{2r}|}{2\pi r^3} \; . $$
For example, for a sphere of radius $R$, where $r = R \phi$ and
$|S_r| = 2\pi R \sin(r/R)$. The R2-D2 formula gives
$4 \pi R \sin(r/R) - 2 \pi R \sin(2r/R) = 2 \pi r^3/R^2 - 2\pi r^5/(4 R^4) + ...$ so that $K = 1/R^2$.
Similarly, for a hyperbolic plane with $|S_r| = 2 \pi R \sinh(r/R)$, we get $K = -1/R^2$.

\paragraph{}
If $C$ is the boundary of a planar region $G$ and $|S_r|$ is the arc length of $S_r$,
the signed curvature is given by the {\bf R2-D2} formula  for $v \in C$ 
$$   K(v) = \lim_{r \to 0} \frac{ 2 |S_r(v)| - |S_{2r}(v)| } {2 r^2} \; . $$
For a disk of radius $R$ one has $|S_r(v)| = 2 r \arccos(r/(2R))$ for a point
on the boundary,
so that $(2 |S_r(v)| - |S_{2r}(v)|)/(2r^2) = 1/R + O(r^2)$ is also intrinsic and 
parameter independent.

\begin{figure}[!htpb]
\scalebox{0.55}{\includegraphics{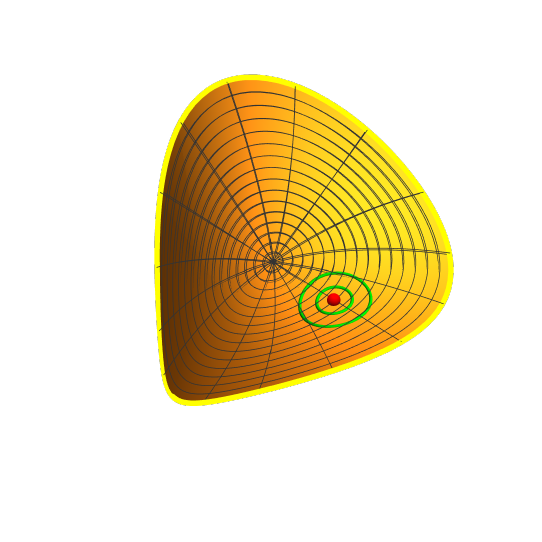}}
\scalebox{0.55}{\includegraphics{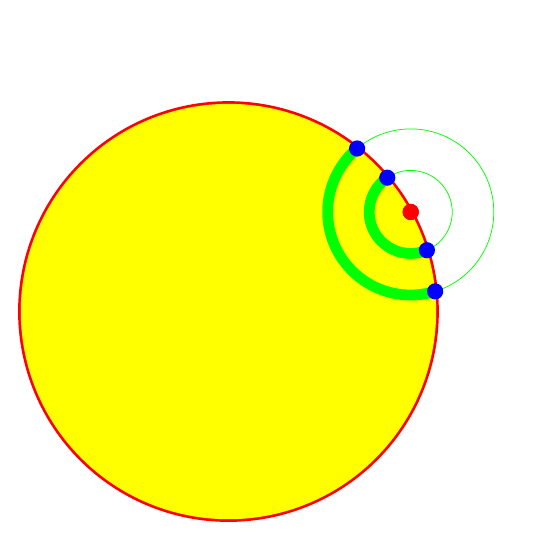}}
\label{rp4}
\caption{
The R2-D2 formula is a classical expression for the curvature of a 2-dimensional
manifold. It holds in the interior and on the boundary. Unlike the classical
{\bf Bertrand-Diquet-Puiseux formula} \cite{BergerPanorama} is
$K = \lim_{t \to 0} \frac{3 (2\pi r - |S_r|)}{r^3 \pi}$ the R2-D2 formulas do
not refer the arc length $2\pi r$ in the flat case.
}
\end{figure}

\paragraph{}
In the discrete, Puiseux fails for general 2-manifolds already. But we can 
get something related now using curvatures on k-forms: 
it follows from the edge curvature that for a 2-manifold we have the curvature
$$   K(v) = 1-\frac{d(v)}{6} + \frac{1}{2}[-1+\sum_{w \in S(v)} \frac{1}{d(w)}] \;  $$
on vertices. A similar curvature appear in \cite{Gromov87}.
Since $\sum_v 1-\frac{d(v)}{6} = \chi(G)$ and $\sum_v K(v) = \chi(G)$ we see
$\sum_v \sum_{w \in S(v)} \frac{1}{d(w)} = V$. Compare the
{\bf Euler handshake formula} $\sum_{v \in V} \sum_{w \in S(v)} 1  = 2E$ which 
by the way is the Gauss-Bonnet formula for the valuation $E(G) = |E|$. 
\footnote{Euler handshake has been called "fundamental theorem of graph theory". 
It is a Gauss-Bonnet relation for the valuation $f_1(G)$. 
Gauss-Bonnet can be formulated for general valuations. }

\section{Code}

\paragraph{}
The following Mathematica code allows to compute the k-curvatures also after
a wave deformation. The case without deformation is obtained by plugging in $t=0$. 

\begin{tiny}
\lstset{language=Mathematica} \lstset{frameround=fttt}
\begin{lstlisting}[frame=single]
Generate[A_]:=If[A=={},{},Sort[Delete[Union[Sort[Flatten[Map[Subsets,A],1]]],1]]];
Cl[G_]:=Union[Sort[Map[Sort,G]]];L=Length;sig[x_]:=Signature[x];Bin=Binomial;Rk=MatrixRank;
Whitney[s_]:=Cl[Generate[FindClique[s,Infinity,All]]];nu[A_]:=If[A=={},0,L[A]-Rk[A]];
F[G_]:=Module[{l=Map[L,G]},If[G=={},{},Table[Sum[If[l[[j]]==k,1,0],{j,L[l]}],{k,Max[l]}]]];
sig[x_,y_]:=If[SubsetQ[x,y]&&(L[x]==L[y]+1),sig[Prepend[y,Complement[x,y][[1]]]]*sig[x],0];
Dirac[G_]:=Module[{f=F[G],b,d,n=L[G]},b=Prepend[Table[Sum[f[[l]],{l,k}],{k,L[f]}],0];
 d=Table[sig[G[[i]],G[[j]]],{i,n},{j,n}]; {d+Transpose[d],b}];
Hodge[G_]:=Module[{Q,b,H},{Q,b}=Dirac[G];H=Q.Q;Table[Table[H[[b[[k]]+i,b[[k]]+j]],
 {i,b[[k+1]]-b[[k]]},{j,b[[k+1]]-b[[k]]}],{k,L[b]-1}]];           omega[x_]:=(-1)^(L[x]-1);
Betti[s_]:=Module[{G},If[GraphQ[s],G=Whitney[s],G=s];Map[nu,Hodge[G]]];
Fvector[A_]:=Delete[BinCounts[Map[L,A]],1];         Euler[A_]:=Sum[omega[A[[k]]],{k,L[A]}]; 
UnitSphere[s_,v_]:=VertexDelete[NeighborhoodGraph[s,v],v];
OpenStar[G_,x_]:=Select[G,SubsetQ[#,x] &]; deg[G_,x_,k_]:=F[OpenStar[G,x]][[k+1]];
Curvatures[s_,k_,t_]:=Module[{G,A,B,Q},If[GraphQ[s],G=Whitney[s],G=s];
 If[t>0,Q=First[Dirac[G]];U=Re[MatrixExp[t*I*Q]],U=IdentityMatrix[Length[G]]];
 A=Select[G,(L[#]==k+1)&]; If[L[Union[Flatten[A]]]<L[Union[Flatten[G]]],Print["No cover"]]; 
 Table[y=A[[l]]; Sum[x=G[[i]];omega[x]*U[[i,i]]*If[SubsetQ[y,x],1/deg[G,x,k],
                              If[SubsetQ[x,y],1/Bin[L[x],k+1],0]],{i,L[G]}],{l,L[A]}]];
Curvatures[s_,k_]:=Curvatures[s,k,0]; 
s = RandomGraph[{20, 50}];{Total[Curvatures[s,1]], Euler[Whitney[s]]}
s = ToGraph[Whitney[ToGraph[Whitney[CompleteGraph[3]]]]];
Table[CurvatureGemPlot[s, 1.0*k/4], {k,1,6}]; 
\end{lstlisting}
\end{tiny}

\paragraph{}
Lets look at some 3 manifolds. We start with a non-standard 3-sphere, the homology 
sphere. It can be found on \cite{Manifoldpage}. See also \cite{Lutz1999,BjoernerLutz}. 

\begin{tiny}
\lstset{language=Mathematica} \lstset{frameround=fttt}
\begin{lstlisting}[frame=single]
s=HomologySphere=Generate[
{{0,1,5,11},{0,1,5,17},{0,1,7,13},{0,1,7,19},{0,1,9,15},{0,1,9,21},{0,1,11,13},
{0,1,15,19},{0,1,17,21},{0,2,6,12},{0,2,6,18},{0,2,7,13},{0,2,7,19},{0,2,10,16},
{0,2,10,22},{0,2,12,13},{0,2,16,19},{0,2,18,22},{0,3,8,14},{0,3,8,20},{0,3,9,15},
{0,3,9,21},{0,3,10,16},{0,3,10,22},{0,3,14,22},{0,3,15,16},{0,3,20,21},{0,4,5,11},
{0,4,5,17},{0,4,6,12},{0,4,6,18},{0,4,8,14},{0,4,8,20},{0,4,11,12},{0,4,14,18},
{0,4,17,20},{0,11,12,13},{0,14,18,22},{0,15,16,19},{0,17,20,21},{1,2,5,11},{1,2,5,17},
{1,2,8,14},{1,2,8,20},{1,2,10,16},{1,2,10,22},{1,2,11,14},{1,2,16,20},{1,2,17,22},
{1,3,6,12},{1,3,6,18},{1,3,7,13},{1,3,7,19},{1,3,8,14},{1,3,8,20},{1,3,12,20},{1,3,13,14},
{1,3,18,19},{1,4,6,12},{1,4,6,18},{1,4,9,15},{1,4,9,21},{1,4,10,16},{1,4,10,22},{1,4,12,16},
{1,4,15,18},{1,4,21,22},{1,11,13,14},{1,12,16,20},{1,15,18,19},{1,17,21,22},{2,3,5,11},
{2,3,5,17},{2,3,6,12},{2,3,6,18},{2,3,9,15},{2,3,9,21},{2,3,11,15},{2,3,12,21},{2,3,17,18},
{2,4,7,13},{2,4,7,19},{2,4,8,14},{2,4,8,20},{2,4,9,15},{2,4,9,21},{2,4,13,21},{2,4,14,15},
{2,4,19,20},{2,11,14,15},{2,12,13,21},{2,16,19,20},{2,17,18,22},{3,4,5,11},{3,4,5,17},
{3,4,7,13},{3,4,7,19},{3,4,10,16},{3,4,10,22},{3,4,11,16},{3,4,13,22},{3,4,17,19},
{3,11,15,16},{3,12,20,21},{3,13,14,22},{3,17,18,19},{4,11,12,16},{4,13,21,22},{4,14,15,18},
{4,17,19,20},{11,12,13,23},{11,12,16,23},{11,13,14,23},{11,14,15,23},{11,15,16,23},
{12,13,21,23},{12,16,20,23},{12,20,21,23},{13,14,22,23},{13,21,22,23},{14,15,18,23},
{14,18,22,23},{15,16,19,23},{15,18,19,23},{16,19,20,23},{17,18,19,23},{17,18,22,23},
{17,19,20,23},{17,20,21,23},{17,21,22,23}}];
Curvatures[s,0]
Curvatures[s,1]
Curvatures[s,2]
Curvatures[s,3]
\end{lstlisting}
\end{tiny}

\paragraph{}
The f-vector of the homology manifold is $(24, 154, 260, 130)$. 
The 24 vertex curvatures $K_0$ are constant $0$ as for any $3$-manifold. 
The 154 edge curvatures $K_1$ take values in $\{ -67/462, -5/132$, $0, 5/84, 13/22 \}$, 
the 260 face curvatures $K_2$ take values $\{ -1/30, -1/60, -1/90, 0, 1/18 \}$ 
and the 130 chamber curvatures $K_3$ take values $\{-1/40, -1/90, 0, 7/180\}$. 

\paragraph{}
The following projective 3-space $\mathbb{RP}^3 \sim SO(3)$ is also 
from \cite{Manifoldpage}. 

\begin{tiny}
\lstset{language=Mathematica} \lstset{frameround=fttt}
\begin{lstlisting}[frame=single]
s=RP3=ProjectiveThreeSpace = Generate[
{{1,2,3,4},{1,2,3,5},{1,2,4,6},{1,2,5,6},{1,3,4,7},{1,3,5,7},{1,4,6,7},{1,5,6,7},
{2,3,4,8},{2,3,5,9},{2,3,8,9},{2,4,6,10},{2,4,8,10},{2,5,6,11},{2,5,9,11},{2,6,10,11},
{2,7,8,9},{2,7,8,10},{2,7,9,11},{2,7,10,11},{3,4,7,11},{3,4,8,11},{3,5,7,10},{3,5,9,10},
{3,6,8,9},{3,6,8,11},{3,6,9,10},{3,6,10,11},{3,7,10,11},{4,5,8,10},{4,5,8,11},{4,5,9,10},
{4,5,9,11},{4,6,7,9},{4,6,9,10},{4,7,9,11},{5,6,7,8},{5,6,8,11},{5,7,8,10},{6,7,8,9}}]; 
Curvatures[s, 0]
Curvatures[s, 1]
Curvatures[s, 2]
Curvatures[s, 3]
\end{lstlisting}
\end{tiny}

\paragraph{}
For the projective 3-space the f-vector is $(11, 51, 80, 40)$. The Euler characteristic is $0$, 
the betti vector is $(1,0,0,1)$. 
The 11 vertex curvatures are zero. The 51 curvatures $K_1$ on edges take values in 
$\{-2/15, -1/9, -1/15, 1/30, 2/45 \}$. The 80 curvatures $K_2$ on faces take 
values in $\{ -1/30, -2/105, -11/840, 1/40, 13/420 \}$. The 40 curvatures $K_3$
on chapters take values in $\{ -3/80, -9/280, 1/56, 33/560 \}$. 

\paragraph{}
Here is an example of a small 4-manifold $\mathbb{C}\mathbb{P}^2$. 
\begin{tiny}
\lstset{language=Mathematica} \lstset{frameround=fttt}
\begin{lstlisting}[frame=single]
s=CP2=Generate[{{1,2,3,4,5},{1,2,3,4,7},{1,2,3,5,8},{1,2,3,7,8},{1,2,4,5,6},
{1,2,4,6,7},{1,2,5,6,8},{1,2,6,7,9},{1,2,6,8,9},{1,2,7,8,9},{1,3,4,5,9},
{1,3,4,7,8},{1,3,4,8,9},{1,3,5,6,8},{1,3,5,6,9},{1,3,6,8,9},{1,4,5,6,7},
{1,4,5,7,9},{1,4,7,8,9} {1,5,6,7,9},{2,3,4,5,9},{2,3,4,6,7},{2,3,4,6,9},
{2,3,5,7,8},{2,3,5,7,9},{2,3,6,7,9},{2,4,5,6,8},{2,4,5,8,9},{2,4,6,8,9},
{2,5,7,8,9},{3,4,6,7,8},{3,4,6,8,9},{3,5,6,7,8},{3,5,6,7,9},{4,5,6,7,8},{4,5,7,8,9}}];
Curvatures[s, 0]
Curvatures[s, 1]
Curvatures[s, 2]
Curvatures[s, 3]
\end{lstlisting}
\end{tiny}

\paragraph{}
The f-vector is $(9, 36, 84, 90, 36)$. The Euler characteristic is $3$. The 
Betti numbers are $(1,0,1,0,1)$. 
The $K_0$ curvature is constant $1/3$. The edge curvature $K_1$ is constant $1/12$. 
The face curvatures $K_2$ take value in $\{ -31/140,-1/14,11/140,8/35 \}$. The
3-dimensional curvatures $K_3$ take values in $\{0,1/20,1/30\}$. The 4-dimensional
curvatures finally are all positive again taking values in $1/30,1/10$. 
The $4$-manifold $CP^2$ is an example of a $1/4$ pinched positive curvature manifold.

\paragraph{}
Finally lets mention some questions: What does it mean for $G$ to have $K_1(x)>0$ for all $x$?
Or does it mean for $G$ to have $K_2(x)>0$ for all $x$? We originally looked into the face 
curvature because it felt like something close to 
{\bf sectional curvature}, a notion a bit tricky to 
define in the discrete. As we see in the examples, positive sectional curvature can not 
just be replaced by positive $K_2$ curvature. We had been searching for index expectation 
curvatures $K$ (also in the discrete) which have the property that if a manifold has positive
sectional curvature, then this index  expectation $K$ is positive. That would settle the 
Hopf conjecture. The Gauss-Bonnet-Chern curvature used in the Gauss-Bonnet-Chern theorem
which was developed from 1927 (Hopf) to 1944 (Chern) 
\cite{HopfCurvaturaIntegra,Allendoerfer,Fenchel,AllendoerferWeil,Chern44} works only 
until dimension 4 \cite{Chern1955} but fails for 6-manifolds already \cite{Geroch,Klembeck}. 
One of our motivations to look at curvatures supported on higher dimensional parts of space
was to explore whether there are natural discrete versions of sectional curvature (which is
not just a numerical scheme coming from the continuum) and (a very long shot) to get a
discrete curvature that is positive for sufficiently nice discretizations of positive
curvature manifolds. Index expectation curvatures allow more experimentation. Using 
index expectation curvatures on higher dimensional simplices allows to widen the search. 

\bibliographystyle{plain}

\end{document}